\documentclass[12pt,twoside]{article}
\usepackage{amssymb,amsthm}
\usepackage[namelimits,sumlimits,fleqn]{amsmath}
\usepackage{setspace}
\usepackage[ hmargin={2.2cm,2.4cm}, top=36mm,
    a4paper,pdftex, textheight=230mm,
    headheight=20mm, headsep=10mm, bindingoffset=0.9cm]
    {geometry}
\usepackage[small, margin=20pt]{caption}
\usepackage[usenames,dvipsnames]{color}
\definecolor{darkblue}{rgb}{0.0,0.0,0.6}
\usepackage{float}\restylefloat{figure}
\usepackage{url}
\usepackage[backref=page]{hyperref}
\hypersetup{colorlinks=true,breaklinks=true,
            linkcolor=darkblue,urlcolor=darkblue,
            anchorcolor=darkblue,citecolor=darkblue}

\usepackage{fancyhdr}
\allowdisplaybreaks[4]

\title{Collinear triples and quadruples in Cartesian products in $\mathbb{F}_p^2$}
\author{Giorgis Petridis}
\date{}

\theoremstyle{plain}
\newtheorem{theorem}{Theorem}
\newtheorem{lemma}[theorem]{Lemma}
\newtheorem{proposition}[theorem]{Proposition}

\theoremstyle{definition}

\newtheorem*{notation}{Notation}

\renewcommand*{\backref}[1]{}
\renewcommand*{\backrefalt}[4]{%
    \ifcase #1 (Not cited.)%
    \or        (p.\,#2)%
    \else      (pp.\,#2)%
    \fi}

\newcommand{\F}{\mathbb{F}_p} 
\newcommand{\Fq}{\mathbb{F}_q} 
\newcommand{\cI}{\mathcal{I}} 
\newcommand{\supp}{\mathrm{supp}} 
\newcommand{\ds}{\displaystyle} 
\newcommand{\f}{\frac} 

\begin{document}

\onehalfspacing

\pagenumbering{arabic}

\setcounter{section}{0}

\bibliographystyle{plain}

\maketitle

\begin{abstract}
In this informal note, which has been absorbed in~\cite{GPSumProd}, we combine a recent point-line incidence bound of Stevens and de Zeeuw with an older lemma of Bourgain, Katz and Tao to bound the number of collinear triples and quadruples in a Cartesian product in $\mathbb{F}_p^2$.
\end{abstract}

\section[Introduction]{Introduction}
\label{Intro}

\let\thefootnote\relax\footnotetext{The author is supported by the NSF DMS Grant 1500984.}

Let $A \subseteq \F$ be a set in a prime order finite field of odd characteristic. In this note, which surpasses the result from~\cite{GPProdDiff}, we combine a recent point-line incidence bound of Stevens and de Zeeuw~\cite{SdZ} with an older lemma of Bourgain, Katz and Tao~\cite{BKT2004} to obtain an improved bound for $T(A)$, the number of collinear triples in $A \times A$, for large sets; and an optimal bound for $Q(A)$, the number of collinear quadruples in $A \times A$.

By a collinear triple we mean an ordered triple $(u,v,w) \in (A \times A) \times  (A \times A) \times  (A \times A)$ such that $u$, $v$ and $w$ are all incident to the same line. So, for example, any point $u \in A \times A$ gives rise to the collinear triple $(u,u,u)$. Collinear quadruples are defined similarly. 

The quantities $T(A)$ and $Q(A)$ can be expressed in terms of the incidence function associated with $A \times A$. For a line $\ell \subset \F^2$, $i(\ell)$ equals the number of points in $A \times A$ incident to $\ell.$ Then
\[
T(A) = \sum_{\text{all lines } \ell} i(\ell)^3 \quad \text{ and } \quad Q(A) = \sum_{\text{all lines } \ell} i(\ell)^4.
\]

Let us study in some detail what is known for collinear triples. The contribution to collinear triples coming from the $|A|$ horizontal and the $|A|$ vertical lines incident to $|A|$ points in $A \times A$ is $2|A|^4$. All other collinear triples can be counted by the number of solutions to
\begin{equation}\label{Alg T(A)}
\frac{a_1 - a_2}{a_3 - a_4} = \frac{a_1 - a_5}{a_3-a_6} \neq 0, \; a_i \in A, \, a_3-a_4, a_3 - a_6 \neq 0.
\end{equation}
It follows from this that the expected number of collinear triples of a random set (where elements of $\F$ belong to $A$ independently with probability $|A|/p$) is $\ds \frac{|A|^6}{p} + 2 |A|^4$. This is because for each 5-tuple $(a_1,\dots, a_5)$ there is a unique element $a_6 \in \F$ that satisfies~\eqref{Alg T(A)} and it belongs to $|A|$ with probability $|A|/p$. 

Another interesting example is that of sufficiently small arithmetic progressions. First note that in general $T(A)$ equals the number of solutions to
\begin{equation*}
(a_1 - a_2)(a_3 - a_4) = (a_1 - a_5)(a_3-a_6) \; a_i \in A.
\end{equation*}
plus $O(|A|^4)$. So it equals
\[
\sum_{a_1, a_3 \in A} \sum_{x} f_{a_1,a_3}^2(x) + O(|A|^4),
\]
where $f_{a_1,a_3}(x)$ is the number of ways one can express $x$ as a product $(a_1 -a_2)(a_3-a_4)$ with $a_2,a_4 \in A$. 

Observe that for all $a_1, a_3 \in A$, $\ds \sum_{x} f_{a_1,a_3}(x) = |A|^2$ because each pair $(a_2,a_4) \in A \times A$ contributes 1 to the sum. 

Applying the Cauchy-Schwartz inequality gives
\[
T(A) = \sum_{a_1, a_3 \in A} \sum_{x} f_{a_1,a_3}^2(x) \geq \sum_{a_1,a_3 \in A} \f{\left( \sum_{x} f_{a_1,a_3}(x) \right)^2}{|\supp(f_{a_1,a_3})|} =  \sum_{a_1,a_3 \in A} \f{|A|^4}{|\supp(f_{a_1,a_3})|}.
\]
Now take $A = \{1,\dots, \sqrt{p}/2\} \subset \mathbb{Z}$. Then for all $a_i\in A$ the product $(a_1-a_2)(a_3-a_4)$ $\{1,\dots , p\}$. This means that the support of $f_{a_1,a_3}$ is the same whether $A$ is taken to be a subset of $\mathbb{Z}$ or of $\F$. Ford has shown in~\cite{Ford2008} that the support of $f_{a_1,a_3}$ (in $\mathbb{Z}$ and hence) in $\F$ is $O(|A|^2 / \log(|A|)^\gamma)$ for some absolute constant $\gamma < 1$. Substituting above implies that for $A = \{1,\dots, \sqrt{p}/2\} \subset \F$, we have $T(A) = \Omega(|A|^4 \log(|A|)^\gamma)$.

Because of these two examples, it is natural to expect that the bound
\[
T(A) = O\left( \frac{|A|^6}{p} + \log(|A|) |A|^4 \right)
\]
is correct up to perhaps logarithmic factors. Over the reals, Elekes and Ruzsa observed in~\cite{Elekes-Ruzsa2003} that the above inequality follows from the Szemer\'edi-Trotter point-line incidence theorem~\cite{Szemeredi-Trotter1983}. 

Far less is currently known. As is explained in Section~\ref{Easy T(A)}, it is straightforward to obtain
\[
\left | T(A) - \left(\frac{|A|^6}{p} + 2 |A|^4 \right) \right| \leq p |A|^2.
\]
It follows that if $|A| = \Omega(p^{2/3})$, then $T(A) = \Theta(|A|^6/p).$

In the range $|A| = O(p^{2/3})$, Aksoy Yazici, Murphy, Rudnev and Shkredov~\cite[Proposition~5]{AYMRS}, building on Rudnev's breakthrough result in \cite{Rudnev}, established the bound
\[
T(A) = O(|A|^{9/2}).
\]
While very strong, the result of Aksoy Yazici, Murphy, Rudnev and Shkredov does not improve the range of $|A|$ where $T(A) = O(|A|^6/p)$.

Combining the two bounds above gives
\[
T(A) = O\left( \frac{|A|^6}{p} + |A|^{9/2}\right).
\]

Similar results are true for collinear quadruples. The expected number of collinear quadruples is $\ds \frac{|A|^8}{p^2} + 2|A|^5$. The result of Aksoy Yazici, Murphy, Rudnev and Shkredov implies that $Q(A) = O(|A|^{11/2})$ when $|A| = O(p^{2/3})$. One expects that the correct order of magnitude up to logarithmic factors is
\[
Q(A) = O\left( \frac{|A|^8}{p^2} + \log(|A|) |A|^{5}\right). 
\] 
Once again large random sets and small arithmetic progressions offering (nearly) extremal examples.

We offer an improvement on the know bound for $T(A)$ when $|A| = \Omega(p^{1/2})$ and establish a nearly best possible bound for $Q(A)$.

\begin{theorem}\label{T(A)}
Let $A \subseteq \F$. 
\begin{enumerate}
\item The number of collinear triples in $A \times A$ satisfies 
\[
T(A) = O\left( \frac{|A|^6}{p} + p^{1/2} |A|^{7/2} \right).
\] 
So there is at most a constant multiple of the expected number of collinear triples when $|A| = \Omega( p^{3/5})$. 
\item The number of collinear quadruples in $A \times A$ satisfies 
\[ 
Q(A) = O\left( \frac{|A|^8}{p^2} + \log(|A|) |A|^5 \right),
\] which is optimal up to perhaps logarithmic factors.
\end{enumerate}
\end{theorem}

The proof of the theorem is based on a recent point-line incidence bound for Cartesian products proved by Stevens and de Zeeuw~\cite[Theorem 4]{SdZ} and a lemma of Bourgain, Katz and Tao~\cite[Lemma~2.1]{BKT2004}. A more precise version of Theorem~\ref{T(A)} and applications to sum-product questions in $\F$ are given in~\cite{GPSumProd}.
 
\begin{notation}
We use Landau's notation so that both statements $f = O(g)$ and $g = \Omega(f)$ mean there exists an absolute constant $C$ such that $f \leq C g$ and $f = \Theta(g)$ stands for $f=O(g)$ and $f= \Omega(g)$. The letter $p$ denotes an odd prime, $\F$ the finite field with $p$ elements and $\Fq^2$ the $2$-dimensional vector space over $\F$. For a line $\ell$, $i(\ell)$ represents the number of points in $A \times A$ incident to $\ell.$ 
\end{notation}

\section[The two ingredients]{The two ingredients}

Let us state the two main ingredients of the proof of Theorem~\ref{T(A)} and some straightforward consequences. We begin with the theorem of Stevens and de Zeeuw.

\begin{theorem}[Stevens and de Zeeuw]\label{SdZ}
Let $A \subseteq \F$ and $L$ be a collection of lines in $\F^2$. Suppose that $|A| |L| = O( p^2 )$. The number of point-line incidences between $A \times A$ and $L$ satisfies $ \cI(A \times A , L ) = O(|L|^{3/4} |A|^{5/4})$.
\end{theorem}

Sevens and de Zeeuw very reasonably imposed the additional condition $|A| < |L| < |A|^3$ because in other ranges, the Cauchy-Schwartz point-line incidence bound is better to theirs. For simplicity of argument, we omit the condition. A sanity check that allowing $|L| \leq |A|$ of $|L| \geq |A|^3$ does not hurt.
\begin{enumerate}
\item If $|L| \leq|A|$, we have $\cI(A \times A , L) \leq |L| |A| \leq |L|^{3/4} |A|^{5/4}$.
\item If $|A|^3 \leq|L|$, we have $\cI(A \times A , L) \leq |L|^{1/2} |A|^2 \leq |L|^{3/4} |A|^{5/4}$.
\end{enumerate}

Next we reformulate the lemma of Bourgain, Katz and Tao in terms of the incidence function $i$, c.f.~\cite[Lemma 1]{GPInc}. Sums are over all lines in $\F^2$ and not just those incident to some point of $A \times A$.
\begin{lemma}[Bourgain, Katz and Tao]\label{BKT i}
Let $A \subseteq \F$.
\[
\sum_{\text{ all lines } \ell}  i(\ell)^2 = |A|^4 + p |A|^2.
\]
In particular
\[
\sum_{\text{ all lines } \ell} \left( i(\ell) - \frac{|A|^2}{p} \right)^2 \leq p |A|^2.
\]
\end{lemma}

The simple yet powerful lemma of Bourgain, Katz, Tao was implicitly extended to not necessarily Cartesian product sets by Vinh~\cite{Vinh2011}. The paper~\cite{GPInc} contains other applications.

Next let $M$ be a parameter and set
\begin{equation}\label{LM}
L_M = \{  M <  i(\ell) \leq  2M \}
\end{equation}
to be the collection of lines from $L$ that are incident to between $M$ and $2M$ points in $A \times A$. We begin with an easy consequence of Lemma~\ref{BKT i}.
\begin{lemma}\label{BKT In}
Let $A \subseteq \F$ and $M$ be a real number. Suppose that $M \geq 2 |A|^2 /p$, then the set $L_M$ defined in~\eqref{LM} satisfies $|L_M| \leq 4 p |A|^2 / M^2$.
\end{lemma}
\begin{proof}
The hypothesis $\ds i(\ell) \geq \tfrac{2 |A|^2}{p}$ implies that $\ds i(\ell) - \tfrac{|A|^2}{p} \geq \tfrac{i(\ell)}{2}\geq \tfrac{M}{2}$.

Lemma~\ref{BKT i} now implies
\[
\frac{M^2}{4} |L_M| \leq \sum_{\ell \in L_M} \left( i(\ell) - \frac{|A|^2}{p} \right)^2  \leq \sum_{\text{ all lines } \ell} \left( i(\ell) - \frac{|A|^2}{p} \right)^2 \leq p |A|^2.
\]
The claim follows.
\end{proof}

We now feed this bound to Theorem~\ref{SdZ} to bound $L_M$. The use of Lemma~\ref{BKT In} is a little strange, because we use it to prove that $|L_M|$ is not too big, so we may apply Theorem~\ref{SdZ} and obtain a reasonable bound. The lemma plays a much more crucial role later.

\begin{lemma}\label{L_M}
Let $A \subseteq \F$ and $\ds M \geq 2 |A|^{3/2} / p^{1/2}$ be a real number.  Then he set $L_M$ defined in~\eqref{LM} satisfies
\[
|L_M| = O\left( \frac{|A|^5}{M^4} \right).
\]
In particular, under this hypothesis, 
\[ 
\sum_{\ell \in L_M} i(\ell)^3 =  O\left( \frac{|A|^5}{M} \right) \text{ and } \sum_{\ell \in L_M} i(\ell)^4 =  O(|A|^5).
\]
\end{lemma}

\begin{proof}
The second claim follows by the first because, say, 
\[
\sum_{\ell \in L_M} i(\ell)^3 \leq 8 M^3 |L_M|.
\]

To establish the first, we apply Theorem~\ref{SdZ}. Therefore we must confirm the condition $|A| |L| = O(p^2)$. The hypothesis $M \geq 2 |A|^{3/2} / p^{1/2}$ implies $ M \geq 2 |A|^2 /p$. Lemma~\ref{BKT In} can therefore be applied and in conjunction with the hypothesis $M^2 \geq 4 |A|^3 / p$ gives
\[
|L_M| \leq \frac{4 p |A|^2}{M^2} \leq \f{p^2}{|A|}.
\]
Hence, $|A| |L_M| \leq p^2$ and Theorem~\ref{SdZ} may be applied. It gives
\[
M |L_M| \leq \cI(A \times A , L_M) = O(|A|^{5/4} |L_M|^{3/4}).
\]
The stated bound follows.
\end{proof}

\section[A straightforward bound for T(A)]{A straightforward bound for $T(A)$}
\label{Easy T(A)}

Before proving Theorem~\ref{T(A)}, let us deduce from Lemma~\ref{BKT i} a straightforward bound for $T(A)$.

\begin{proposition}
Let $A\subseteq \F$. Then
\[
\left| T(A) - \left( \f{|A|^6}{p} + 2 |A|^4 \right) \right| \leq p |A|^3. 
\]
\end{proposition}

\begin{proof}
Sums are over all lines in $\F^2$.  We combine the first part of Lemma~\ref{BKT i} with the identity $\sum_{\ell} i(\ell) = (p+1) |A|^2$, which follows from the fact that every point in $A \times A$ is incident to $p+1$ lines in $\F^2$.
\begin{align*}
T(A) 
& = \sum_\ell i(\ell)^3 \\
& = \sum_\ell i(\ell) \left(i(\ell) - \f{|A|^2}{p} \right)^2 + 2 \f{|A|^2}{p} \sum_\ell i(\ell)^2 - \left(\f{|A|^2}{p}\right)^2 \sum_\ell i(\ell) \\
& = \sum_\ell i(\ell) \left(i(\ell) - \f{|A|^2}{p} \right)^2  + \f{|A|^6}{p} + 2 |A|^4 - \f{|A|^6}{p^2}.
\end{align*}
The claim now follows from the fact that $i(\ell) \leq |A|$ and the second part of Lemma~\ref{BKT i}  because
\[
\left| T(A) - \left( \f{|A|^6}{p} + 2 |A|^4 \right) \right| \leq \sum_\ell i(\ell) \left(i(\ell) - \f{|A|^2}{p} \right)^2 \leq |A| \sum_\ell \left(i(\ell) - \f{|A|^2}{p} \right)^2 \leq p|A|^3. 
\]
\end{proof}

We therefore see that to improve the proposition, and hence the range of $|A|$ for which $T(A) = O(|A|^6/p)$, we must show that it is impossible for ``the mass'' of
\[
\sum_\ell \left(i(\ell) - \f{|A|^2}{p} \right)^2 \approx p|A|^2
\] 
to come from lines that satisfy $i(\ell) = \Omega(|A|)$. In other words, we must roughly speaking show that it cannot be the case that $\Omega(p)$ lines are incident to $\Omega(|A|)$ points in $A \times A$. The theorem of Stevens and de Zeeuw guarantees this. In fact Theorem~\ref{SdZ} implies that there are $O(|A|)$ lines incident to $\Omega(|A|)$ points in $A \times A$, which is nearly optimal. To maximise the gain we perform a more careful analysis. 
 
\section[Proof of Theorem~1]{Proof of Theorem~\ref{T(A)}} 

For $T(A)$ we break the sum of the cubes of $i(\ell)$ in three parts: 
\begin{enumerate}
\item Those where $i(\ell)$ is small (these give the term that resembles the expected count).
\item  Those where $i(\ell)$ is of medium size (controlled by Lemma~\ref{BKT In}).
\item  Those  where $i(\ell)$ is large (controlled by Lemma~\ref{L_M} and dyadic decomposition). 
\end{enumerate}
The details are as follows.
\begin{align}\label{T3}
T(A) & 
= \sum_{\text{ all lines } \ell} i(\ell)^3 \nonumber \\
& = \sum_{ i(\ell) \leq \tfrac{2 |A|^2} {p}} i(\ell)^3 + \sum_{ \tfrac{2 |A|^2}{p} \leq i(\ell) \leq \tfrac{2 |A|^{3/2}}{p^{1/2}}} i(\ell)^3  + \sum_{ i(\ell) \geq \tfrac{2 |A|^{3/2}}{ p^{1/2}}} i(\ell)^3.
\end{align}
The first sum is bounded using the identity $\sum_{\ell} i(\ell) = (p+1) |A|^2$:
\[
\sum_{ i(\ell) \leq \tfrac{2 |A|^2} {p}} i(\ell)^3 \leq \f{4 |A|^4}{p^2} \sum_{\ell} i(\ell) = O \left(\f{|A|^6}{p}\right).
\]
The second sum is bounded using Lemma~\ref{BKT In} and the observation that $i(\ell) \geq 2 |A|^2/p$, then $i(\ell) \leq 2 (i(\ell) - |A|^2 /p)$:
\[
\sum_{ \tfrac{2 |A|^2}{p} \leq i(\ell) \leq \tfrac{2 |A|^{3/2}}{p^{1/2}}} i(\ell)^3  \leq \frac{2 |A|^{3/2}}{p^{1/2}} \sum_{\ell} 4 \left(i(\ell) - \frac{|A|^2}{p} \right)^2 = O( p^{1/2} |A|^{7/2}).
\]
The third sum is bounded using Lemma~\ref{L_M} and dyadic decomposition: 
\begin{align*}
\sum_{ i(\ell) \geq \tfrac{2 |A|^{3/2}}{ p^{1/2}}} i(\ell)^3 
& = \sum_{ \tfrac{2 |A|^{3/2}}{ p^{1/2}} \leq 2^{j} \leq \tfrac{|A|}{2}}  \sum_{\ell \in L_{2^j}} i(\ell)^3 \\
& = O \left(  \sum_{ 2^{j} \geq \tfrac{2 |A|^{3/2}}{ p^{1/2}}}  \frac{|A|^5}{2^j} \right) \\
& = O \left( \frac{|A|^5}{\frac{|A|^{3/2}}{p^{1/2}}} \right) = O(p^{1/2} |A|^{7/2}).
\end{align*}
Substituting the three bounds into~\eqref{T3} gives $T\ds (A) = O\left(\f{|A|^6}{p} + p^{1/2} |A|^{7/2}\right)$.

A nearly identical argument works for $Q(A)$.
\begin{align}\label{Q3}
Q(A) & 
= \sum_{\text{ all lines } \ell} i(\ell)^4 \nonumber \\
& = \sum_{ i(\ell) \leq \tfrac{2 |A|^2} {p}} i(\ell)^4 + \sum_{ \tfrac{2 |A|^2}{p} \leq i(\ell) \leq \tfrac{2 |A|^{3/2}}{p^{1/2}}} i(\ell)^4  + \sum_{ i(\ell) \geq \tfrac{2 |A|^{3/2}}{ p^{1/2}}} i(\ell)^4.
\end{align}
The first sum is bounded using the identity $\sum_{\ell} i(\ell) = (p+1) |A|^2$:
\[
\sum_{ i(\ell) \leq \tfrac{2 |A|^2} {p}} i(\ell)^4 \leq \f{8 |A|^6}{p^3} \sum_{\ell} i(\ell) = O \left(\f{|A|^8}{p^2}\right).
\]
The second sum is bounded using Lemma~\ref{BKT In}:
\[
\sum_{ \tfrac{2 |A|^2}{p} \leq i(\ell) \leq \tfrac{2 |A|^{3/2}}{p^{1/2}}} i(\ell)^4  \leq \frac{4 |A|^{3}}{p} \sum_{\ell} 4 \left(i(\ell) - \frac{|A|^2}{p} \right)^2 =O (|A|^{5}).
\]
The third sum is bounded using Lemma~\ref{L_M} (and is this time of greater order of magnitude than the second): 
\begin{align*}
\sum_{ i(\ell) \geq \tfrac{2 |A|^{3/2}}{ p^{1/2}}} i(\ell)^3 
& = \sum_{ \tfrac{2 |A|^{3/2}}{ p^{1/2}} \leq 2^{j} \leq \tfrac{|A|}{2}}  \sum_{\ell \in L_{2^j}} i(\ell)^3 \\
& = O \left(  \sum_{ 2^{j} \leq |A|}  |A|^5 \right) \\
& = O(\log(|A|) |A|^5).
\end{align*}
Substituting the three bounds into~\eqref{Q3} gives $\ds Q(A) = O\left(\f{|A|^8}{p^2} + \log(|A|) |A|^5 \right)$.

\phantomsection

\addcontentsline{toc}{section}{References}

\bibliography{all}

\vskip 0.5cm

\hspace{20pt} Department of Mathematics, University of Georgia, Athens, GA, USA.

\hspace{20pt} \textit{Email address}: \href{mailto:giorgis@cantab.net}{giorgis@cantab.net}

\end{document}